\documentclass[12pt]{article}

\usepackage{natbib}
\usepackage{graphicx}
\usepackage{epsfig}
\usepackage{epstopdf}
\usepackage{amssymb}
\usepackage{amsthm}

\def\U{{\cal U}}

\def\F{{\cal F}}

\def\G{{\cal G}}

\def\D{{\cal D}}
\def\L{{\cal L}}

\newcommand{\RR}{\mbox{I\hspace{-.08cm}R}}

\theoremstyle{plain}
\newtheorem{theorem}{Theorem}

\newtheorem{proposition}[theorem]{Proposition}

\theoremstyle{definition}
\newtheorem{definition}{Definition}

\theoremstyle{remark}

\begin{document}

\title{An Improved Two-Party Negotiation Over Continues Issues Method Secure Against Manipulatory Behavior}

\author{\bfseries\itshape Luca Barzanti\thanks{E-mail address: luca.barzanti@unibo.it} \ and Marcello Mastroleo\thanks{E-mail address: marcello.mastroleo@unibo.it} \\
Department of Mathematics for \\ Economic and Social Sciences\\ University of Bologna \\ Italy}
\date{}
\maketitle

\begin{abstract}
This contribution focuses on two-party negotiation over continuous issues. We firstly prove two drawbacks of the jointly Improving Direction Method (IDM), namely that IDM is not a Strategy-Proof (SP) nor an Information Concealing (IC) method. Thus we prove that the concurrent lack of these two properties implies the actual non-efficiency of IDM. Finally we propose a probabilistic method which is both IC and stochastically SP thus leading to efficient settlements without being affected by manipulatory behaviors.
\end{abstract}



\section{Introduction}
To accomplish suitable settlements, several methods have been proposed in the literature and in practice, in particular the Improving Direction Method (IDM) allows Pareto efficient settlements while negotiating in several continuous contexts. IDM was firstly introduced in \citet{740667} for two-party negotiation over continuous issues, then generalized in \citet{Ehtamo200154} to suite multi-party contexts, finally in \citet{Harri2001} is proven its wide generality.

In Section \ref{Sec:2} we briefly introduce IDM then in Subsection \ref{Sec:3} we show that its deterministic nature constitutes a weakness that can be exploited by a party against the other thus compromising the optimality of the ending settlement. Finally a probabilistic Non Informative Negotiation (NIN) method is proposed, in Section \ref{Sec:4}, so to overcome the IDM's drawbacks. Conclusions are expressed in Section \ref{Sec:5}.

\section{Actual Inefficiency of Two-Party Improving Direction Method}\label{Sec:2}

As firstly presented in \cite{740667}, IDM is thought in relation to the concept of single negotiating text (SNT) which was proposed by R. Fisher during the Camp David negotiations in 1978 (see \cite{FisherCD, FisherUryCD, RaiffaCD}). Following SNT scheme, IDM iterates a refinement algorithm into a feasible real domain until the Pareto frontier is reached. 

We are going to recall IDM in the measure it will be useful for what follows. Thus for a complete description see \cite{740667,Ehtamo200154,Harri2001}. 

Let the {\it negotiation domain} $\D$ be a convex closed subset of $\RR^n$ and let $u_1$, $u_2$ be the utility functions of party 1 and 2 respectively, which are assumed to be differentiable on $\D$, then we can represent IDM negotiation as a {\it fixed point algorithm} which iterates recursively a map $IDM$ consisting of a {\it neighborhood exploring} step $\G$, that addresses the negotiation in the right direction, which is followed by a {\it straight going} step $\L$, that aims to exploit the previously identified direction as far as possible. 

More precisely, $\G(u_1,u_2,x_t)$ is the bisector of the angle formed by the gradients of $u_1(x_t)$ and $u_2(x_t)$:
\begin{equation}\label{Eq:G}
\G(u_1,u_2,x_t) = \frac{\nabla u_1(x_t)} {2||\nabla u_1(x_t)||} + \frac{\nabla u_2(x_t)}{2||\nabla u_2(x_t)||}.
\end{equation}
Whereas $\L$ is the maximum step that can be done in $\G(u_1,u_2,x_t)$ direction without penalizing any of the parties, thus if $\lambda_1$ and $\lambda_2$ are the maximum movements that party 1 and 2 respectively want to do and $\lambda^* = \min(\lambda_1, \lambda_2)$ then, 
\begin{equation}\label{Eq:L}
\L(x_t,\G(u_1,u_2,x_t)) = x_t + \lambda^*_t\cdot \G(u_1,u_2,x_t).
\end{equation}

Note that the algorithm needs an exchange of information between parties and mediator, as authors highlights  in \cite{740667}, ``decision makers reveal minimal private information to the mediator'' which is mostly represented by the local evaluation of utility gradients.

\subsection{Actual Inefficiency of IDM}\label{Sec:3}

The deep analysis of actual exchanged information leads us to identify IDM's actual weaknesses.

\begin{definition} A negotiation method is {\it Strategy-Proof} (SP) if and only if negotiating with the real utility function gains to a party the best ending settlement, regardless the specific utility profile of the others parties.
\end{definition}

\begin{definition}
A negotiation method is {\it Information-Concealing} (IC) if and only if each party is not able to understand other ones utilities by the mediator announcements.
\end{definition}

We prove that IDM fails to be both SP and IC, then its actual inefficiency as the consequence of the conjunction of these two lacks. Therefore we firstly need to exhibit an example of negotiation domain within IDM does not hold any of the two properties. 

We are going to use the abstract negotiation domain whose instance (namely the {\it ``fishing right''} domain described in \citep{FRutilities}) was used in \cite{740667, Ehtamo200154, Harri2001} as the example of IDM effectiveness. Let us, thus, consider for $k\geq0$, the triangular domain 
$$\D = \{(x_1,x_2)\in\RR^2\;|\;x_1,x_2\geq 0\; \wedge\; x_1+x_2\leq k \},$$ 
and the class of utility functions 
$$\U = \{u:\D\rightarrow\RR\;|\;u=\alpha_1\log(x_1)+\alpha_2\log(x_2)+\alpha_3\log(k-x_1-x_2)\wedge \alpha_i\geq0\}.$$
Let us now assume that there is perfect competition between the two negotiating parties $P_1$ and $P_2$, thus the first one has utility $u_1\in\U$ with parameters $(1,0,\beta_1)$ while the latter has utility $u_2\in\U$ with parameters $(0,1,\beta_2)$.

Within this domain, the points $B_1 \equiv (k/(1+\beta_1),0)$ and $B_2 \equiv (0,k/(1+\beta_2))$ correspond to the best utility of $P_1$ and $P_2$ respectively, while the Pareto efficient frontier $\F$ associated to $u_1$ and $u_2$ is the line segment $\overline{B_1B_2}$. Thus according to its optimality, for any $x_0\in\D$, IDM converges on a point $IDM(u_1,u_2,x_0)\in\F$. Moreover we know by the characterization of the efficient frontier that, for all $x\in\F$, $\nabla u_i(x) \parallel \partial\F(x)$.

\begin{theorem}\label{Thm:1}
IDM is not SP.
\end{theorem}
\begin{proof}
Let us suppose by absurd that IDM is SP, thus for all $\gamma \geq 0$, if $v_\gamma=\log(x_1)+\gamma\log(k-x_1-x_2)$ then, by looking at $P_1$, 
\begin{equation}\label{Eq:Max}
u_1(IDM(v_\gamma,u_2,x_0)) \leq u_1(IDM(u_1,u_2,x_0)),
\end{equation}
where the equality corresponds to the choice $\gamma = \beta_1$. In particular by fixing $k = 10$, maximum utility of $P_1$ in $(2,0)$, of $P_2$ in $(0,3)$ and starting by $x_0 \equiv (5, 4)$ the negotiation sops on the Paretian at $\bar{x}\equiv(1.1410,1.2884)$ and gaining $P_1$ an ending utility of $u_1(\bar{x}) = 8.2290$. Under the same assumptions if $P_1$ declares $v_{7/3}\neq u_1 (= v_{4})$ then the negotiation ends in  $x'\equiv(1.7435, 1.2565)$ corresponding to an ending utility $u_1(x') = 8.3395 > u_1(\bar{x})$ which contradicts the hypothesis.
\end{proof}

\begin{theorem}\label{Thm:2}
IDM is not IC.
\end{theorem}

\begin{proof}
Within the same negotiation domain used in the proof of Theorem \ref{Thm:1} once, without loss of generality, $P_1$ observes the announced $\G(u_1,u_2,x_0)$, he can calculate 
$$v \equiv \frac{\nabla u_2(x_0)}{||\nabla u_2(x_0)||} = 2\G(u_1,u_2,x_0) - \frac{\nabla u_1(x_0)}{||\nabla u_1(x_0)||}$$
by inverting equation \ref{Eq:G}. Knowing $v$, $P_1$ can find the parameter $\beta_2$ which characterizes $u_2$ by solving $(\nabla u_2(x_0), (v_2,v_1)) = 0$ which is linear in $\beta_2$.
\end{proof}

Note that if the estimation of more than one parameter is needed, then $P_1$ can slow the IDM convergence as down as he wants by simply declaring in the $\L$ stage a $\lambda_1$ step sufficiently small, thus he can collect all the $\nabla u_2$ samples he needs to uniquely understand $u_2$.

\begin{theorem}
IDM is actual inefficient when there is a one to one correspondence between utility profiles and Pareto frontiers.
\end{theorem}

\begin{proof}
Each party can exploit the IDM's lack of IC to retrieve other one utility and to take advantage of IDM's lack of SP, so if $u^1_1$ and $u^1_2$ are the parties true utilities while $u^2_1$ and $u^2_2$ are the strategical utilities, then there are four possible ending settlements, according to $P_i$ using $u^1_i$ or $u^2_i$, each of them on a different Pareto efficient frontier $\F_{u^j_1,u^k_2}$. By construction of both $u^2_1$ and $u^2_2$, it results that $u^1_1(IDM(u^2_1,u^1_2,x_0)) > u^1_1(IDM(u^1_1,u^1_2,x_0))$ and $u^1_2(IDM(u^1_1,U^2_2,x_0)) > u^1_2(IDM(u^1_1,u^1_2,x_0))$.

The last two inequalities suffice to say that for both $P_1$ and $P_2$ the action of using their own real utility function is not a dominant strategy, thus $IDM(u^1_1,u^1_2,x_0)$ is not the solution of the non cooperative strategic form game associated to the outcomes.
Moreover whenever a change of utility function corresponds to a change of the Pareto efficient frontier, like in the presented negotiation domain, IDM results actual inefficient in consequence of the fact that if $(j,k)\not\equiv(1,1)$ then $IDM(u^j_1,u^k_2,x_0)\notin\F_{u^1_1,u^1_2}$.
\end{proof}

Figure \ref{Fig:IDM} shows IDM lacks in the triangular domain used in the proof of Theorem \ref{Thm:1}. The four trajectories correspond (from left to right and from the top to the bottom) to $(u_1^1,u_2^2)$, $(u_1^2,u_2^2)$, $(u_1^1,u_2^1)$ and $(u_1^2,u_2^1)$. Inefficiency is then confirmed by looking at the strategic form game associated to the ending utilities which is equivalent to the {\it ``Prisoner's Dilemma''}, whose dominant strategy solution $(u_1^2,u_2^2)$ is well-known suboptimal.

\begin{figure}[htb]
\begin{center}
\includegraphics[width=100mm]{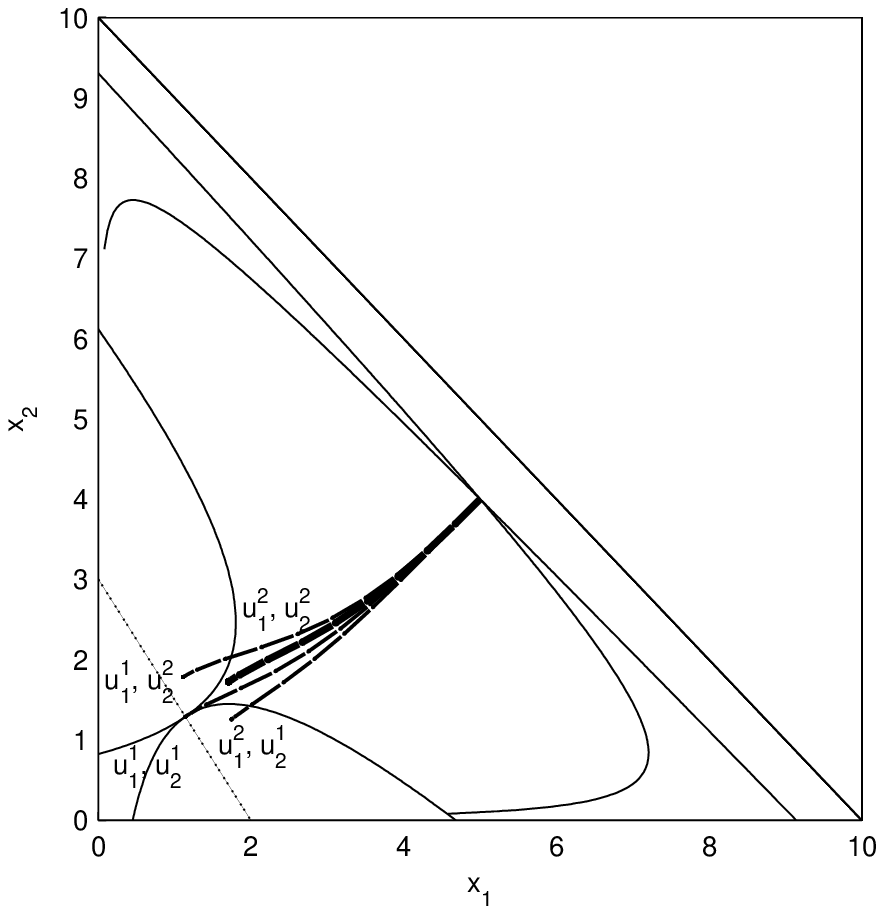}

\begin{tabular}{|c|c|c|}
\hline
 & $u_1^1$ & $u_1^2$ \\
\hline
$u_2^1$ & 8.2290 4.9767 & 8.3395 4.7688\\
\hline
$u_2^2$ & 7.9501 5.1536 & {\bf 8.0642 4.9368}\\
\hline
\end{tabular}

\end{center}
\caption{Example of IDM inefficiency. The thicker trajectory corresponds to $\{u_1^2,u_2^2\} \equiv (8.0642,4.9368)$, {\it i.e.} the dominant strategy solutions  of the {\it ``Prisoner's Dilemma''} like strategic form game, which results by looking at the payoffs table.}\label{Fig:IDM}
\end{figure}

\section{Two-Party Negotiation Method Secure\\Against Manipulatory Behavior.}\label{Sec:4}

In this Section we propose a probabilistic Non Informative Negotiation method (NIN) that follows the SNT scheme but it avoids mutual information leaks. Hence NIN is a fixed point algorithm whose iterated map $NIN:\D\rightarrow\D$ is composed, like IDM's one, by a  {\it neighborhood exploring} step $\G$ and a {\it straight going} step $\L$ but, differently from IDM, $\G$ is based on stochastic answers.

Thereby step by step, each party $P_i$ chooses a random vector $v_i$ which is picked according to a {\it secret} probability distribution with mean $\nabla u_i(x_t)$, then $\G$ is computed as the bisector of the two {\it randomly chosen} directions rather than gradients
\begin{equation}
\G(x_t) = \frac{v_1} {2||v_1||} + \frac{v_1}{2||v_2||}.
\end{equation}

Due to the stochastic nature of parties answers, the condition $NIN(x_t) = x_t$ does not guarantee that $x_t$ is on the Pareto frontier $\F$. Nevertheless the following proposition holds.

\begin{proposition}\label{Prom:NINstop}
If $x$ is on the Pareto frontier $\F$ of $u_1$ and $u_2$ then $$P(NIN(x) = x) = 1,$$ independently of the secret probability functions $\mu_1$ and $\mu_2$ of $P_1$ and $P_2$ respectively. 
\end{proposition}

\begin{proof}
From the characterization of $\F$ follows that a movement in the direction $d$, for any $d\in\RR^2$, penalizes at least one party, {\it e.g.} $P_1$. Thus $P_1$ would declare at the $\L$ step $\lambda_1 = 0$ if the direction is $d$. Thus by the arbitrary of $d$ follows that $NIN(x) = x$ always on $\F$.   
\end{proof}

Proposition \ref{Prom:NINstop} can be used to draw the $NIN$ stop condition, {\it i.e.} to fix a maximum number $M$ of consecutive tries to find a new settlement point above which the negotiation is ended. In particular the following proposition leads to $NIN$'s stochastic efficiency.

\begin{proposition}\label{Prop:Exp}
The probability to stop $NIN$ negotiation earlier that reaching the Pareto frontier decays to zero exponentially on $M$.
\end{proposition}

\begin{proof}
Regardless the actual sub optimal settlement $x\in\D$, if the probability to have $\G(x)$ in the feasible region of $x$ is $\epsilon$ then the probability to fail $M$ in a row tries to improve $x$ is exactly $(1-\epsilon)^M$.
\end{proof}

The NIN method can be summarized as it follows:
\begin{enumerate}
\item Start at the pre-agreed settlement $x_0$ and set $t = 0$;
\item $r\leftarrow 0$;
\item while $r < M$ do
	\begin{enumerate}
	\item if $NIN(x_t) \neq x_t$ then
	\begin{enumerate}
	\item $x_{t+1} \leftarrow NIN(x_t)$;
	\item $t \leftarrow t+1$;
	\item $r \leftarrow 0$;
	\end{enumerate}
	\item else
	\begin{enumerate}
	\item $r \leftarrow r+1$;
	\end{enumerate} 
	\end{enumerate}
\item return $x^* \equiv x_t$. 
\end{enumerate}

Note that the ``noise'' introduced by the stochastic component of NIN makes both useless the a priori mutual knowledge and impossible the process of deterministically retrieve the other one utility function. Nevertheless an adequate choice of the stopping bound $M$ leads exponentially close to the Pareto frontier as can be seen in Figure \ref{Fig:NINdist}, where the Mean Relative Error 
\begin{equation}
MRE = \sum\limits_{i=1}^n \frac{d(\F,NIN(u_1,u_2,x_0))}{n\cdot d(\F,x_0)}
\end{equation}
is plotted varying $M$, and in Figure \ref{Fig:M} where the Relative Error frequencies are sampled for the first values of $M$. In particular the exponentially decreasing shape of both figures reflects the Proposition \ref{Prop:Exp} and for $M$ as little as 5 the $MRE$ is 0.0161 meaning that NIN reduced the starting distance at the 1.61\% of the original one. Moreover for any $M\geq1$ NIN ends closer to $\F$ than any strategical path of IDM; Figure \ref{Fig:NIN} shows by means of 10 random samples how NIN ends very close to the Pareto frontier, when $M = 5$.

\begin{figure}[htb]
\begin{center}
\includegraphics[width=100mm]{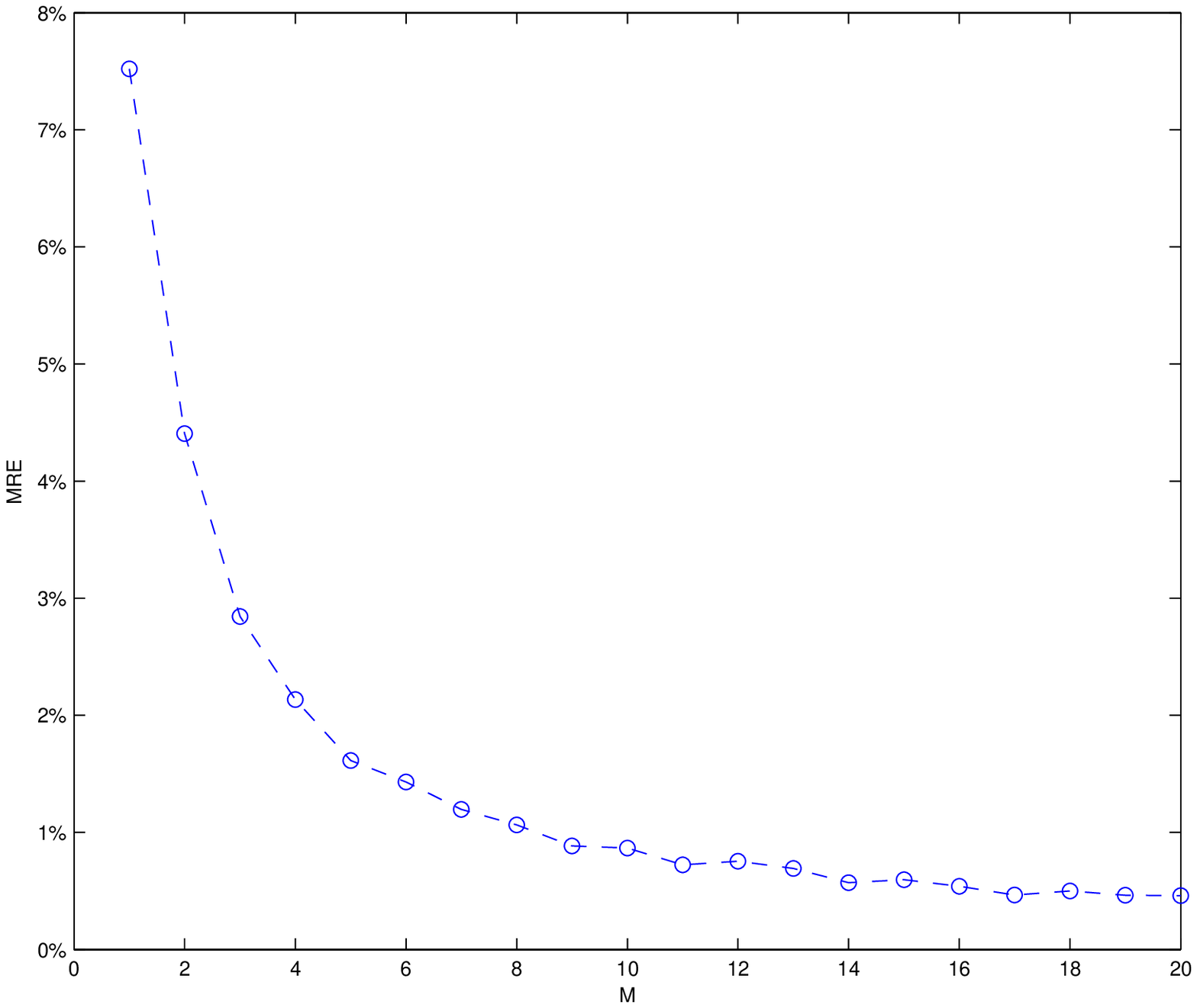}
\end{center}
\caption{MRE varying $M$.}\label{Fig:NINdist}
\end{figure}

\begin{figure}[htb]
\begin{center}
\includegraphics[width=100mm]{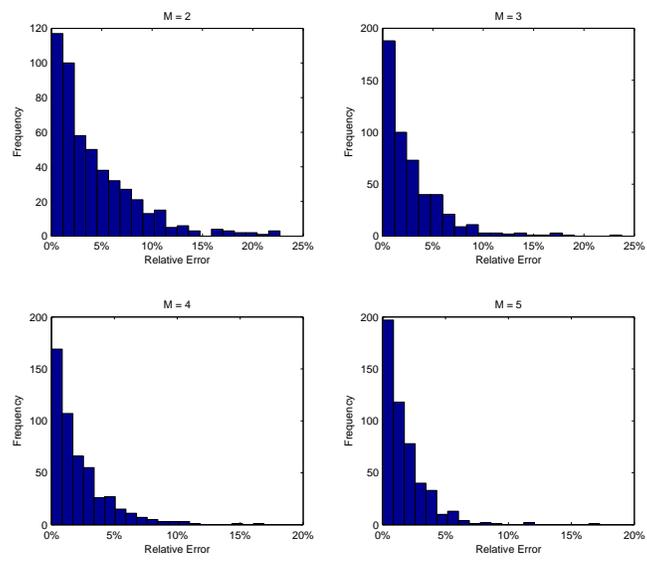}
\end{center}
\caption{Relative Error frequencies for different values of $M$ over 500 samples.}\label{Fig:M}
\end{figure}

\begin{figure}[htb]
\begin{center}
\includegraphics[width=100mm]{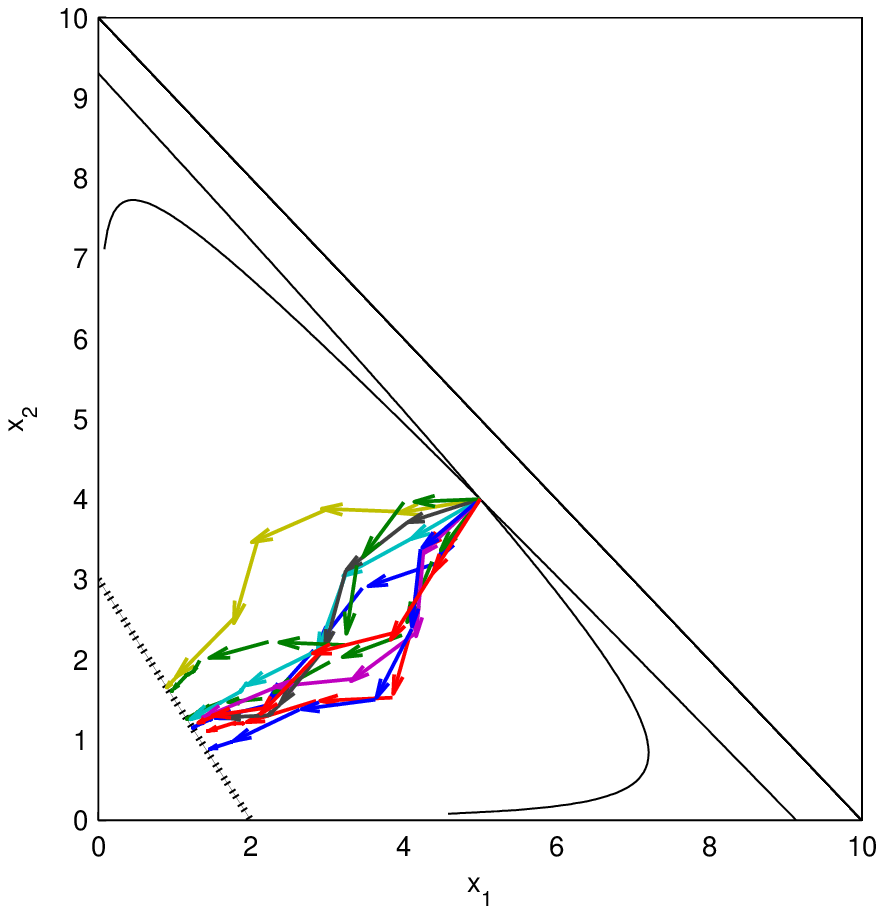}
\end{center}
\caption{10 NIN trajectories with $M = 5$ in the same example of Figure \ref{Fig:IDM}.}\label{Fig:NIN}
\end{figure}

\section{Conclusion}\label{Sec:5}
In this contribution we highlighted two main drawbacks of the IDM negotiation method which affect its global efficiency, namely that IDM is sensible to information of other party utility and that IDM itself conveys mutual knowledge. We then proposed a probabilistic method that overcomes these practical limitations without loosing the efficiency of the ending settlement.


\bibliographystyle{plainnat}
\bibliography{MPN}
\end{document}